\documentclass[reqno,12pt,letterpaper]{amsart}
\usepackage[proof]{zlmacros}

\title{Internal waves in 2D domains with ergodic classical dynamics}
\email{yves.colin-de-verdiere@univ-grenoble-alpes.fr}
\address{Universit\'e Grenoble-Alpes,
Institut Fourier,
 Unit{\'e} mixte
 de recherche CNRS-UGA 5582,
 BP 74, 38402-Saint Martin d'H\`eres Cedex (France)}
\author{Yves Colin de Verdi{\`e}re and Zhenhao Li}
\email{zhenhao@mit.edu}
\address{Department of Mathematics, Massachusetts Institute of Technology, Cambridge, MA 02139 (U.S.)}

\def\gl{\lambda}

\begin{document}

\maketitle
\begin{abstract}
    We study a model of internal waves in an effectively 2D aquarium under periodic forcing. In the case when the underlying classical dynamics has sufficiently irrational rotation number, we prove that the energy of the internal waves remains bounded. This involves studying the spectrum of a related 0-th order pseudodifferential operator at spectral parameters corresponding to such dynamics. For the special cases of rectangular and elliptic domains, we give an explicit spectral description of that operator.
\end{abstract}
\section{Introduction}
Below the surface layers of the ocean, the density field can be be approximated by a stable-stratified field. This means that the density depends only on depth and increases slowly with it. A standard model for internal waves is given by 
considering linear perturbations of such stable-stratified fluids, and is a central topic in oceanography. These perturbations occur naturally and can arise mechanically or thermodynamically. For a more complete introduction to the physics behind internal waves, see Mass \cite{Maas_survey} and Sibgatullin--Ermanyuk \cite{SE_survey}.


In this paper, we consider an open, bounded, and simply-connected domain $\Omega \subset \R^2_{x_1, x_2}$ with a smooth boundary. Internal waves are modeled by the equation
\begin{equation}\label{eq:internal_waves}
    (\partial_t^2 \Delta + \partial_{x_2}^2) u = f(x) \cos (\lambda t), \quad u|_{t = 0} = \partial_t u|_{t = 0} = 0, \quad u|_{\partial \Omega } = 0
\end{equation}
with $\lambda \in (0, 1)$ and $f \in C^\infty(\overline \Omega; \R)$. This is the Poincar\'e equation \cite{Po}, also called the Sobolev equation \cite{sobolev}. This problem comes from the study of internal waves in a 2D aquarium with a constant Brunt-Va\"is\"al\"a frequency which we take equal to $1$. The solution $u$ represents the \textit{stream function} of the fluid velocity, meaning the velocity field is given by 
\begin{equation}\label{eq:vf}
    \mathbf v = (\partial_{x_2} u, -\partial_{x_1} u).
\end{equation}
Then \eqref{eq:internal_waves} can be interpreted as the evolution of the stream function under periodic forcing in the interior of the domain with forcing profile $f$. The Dirichlet boundary condition is simply saying that velocity of the fluid near the boundary must be tangent to the boundary, i.e. no forcing from the boundary. Let $g \in C^\infty(\partial \Omega; \R)$, and let $\tilde g|_{\partial \Omega} =  g$ and $\Delta \tilde g = 0$. Then the boundary forced equation 
\[(\partial_t^2 \Delta + \partial_{x_2}^2) u = 0, \quad u|_{t = 0} = \tilde g, \quad \partial_t u|_{t = 0} = 0, \quad u|_{\partial \Omega } = g(x) \cos(\lambda t)\]
can be easily reduced to \eqref{eq:internal_waves} by replacing $u(x, t)$ with $u(x, t) - \tilde g(x) \cos (\lambda t)$.

\begin{figure}
    \centering
    \begin{subfigure}{.16\textwidth}
        \centering
        \includegraphics[width=\linewidth]{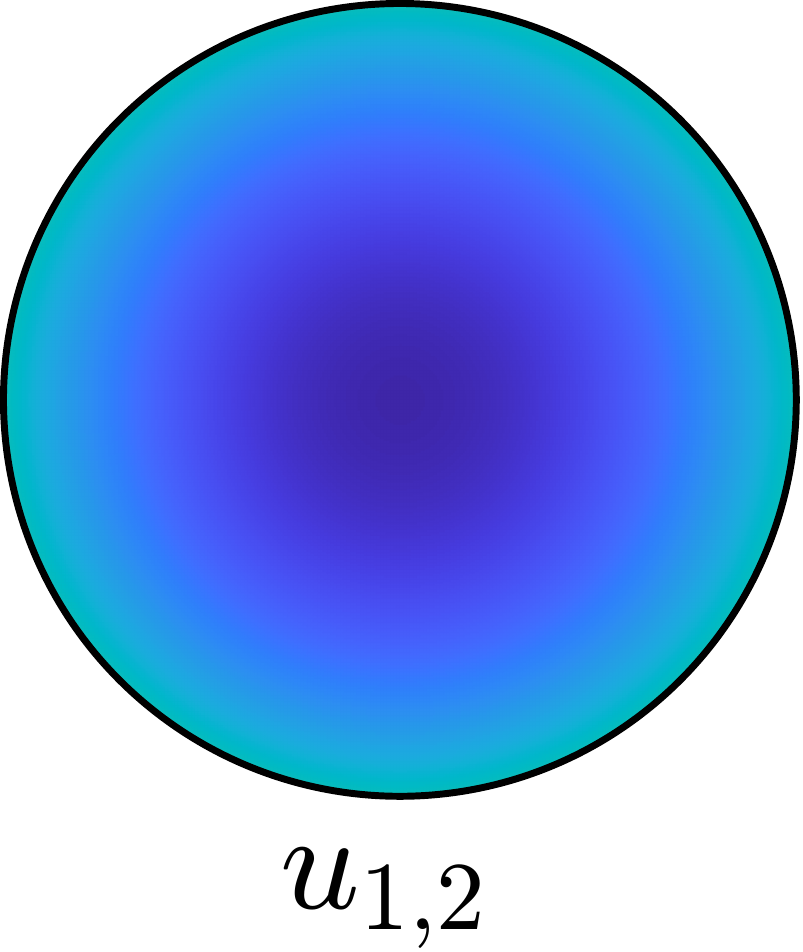}
    \end{subfigure}
    \begin{subfigure}{.16\textwidth}
        \centering
        \includegraphics[width=\linewidth]{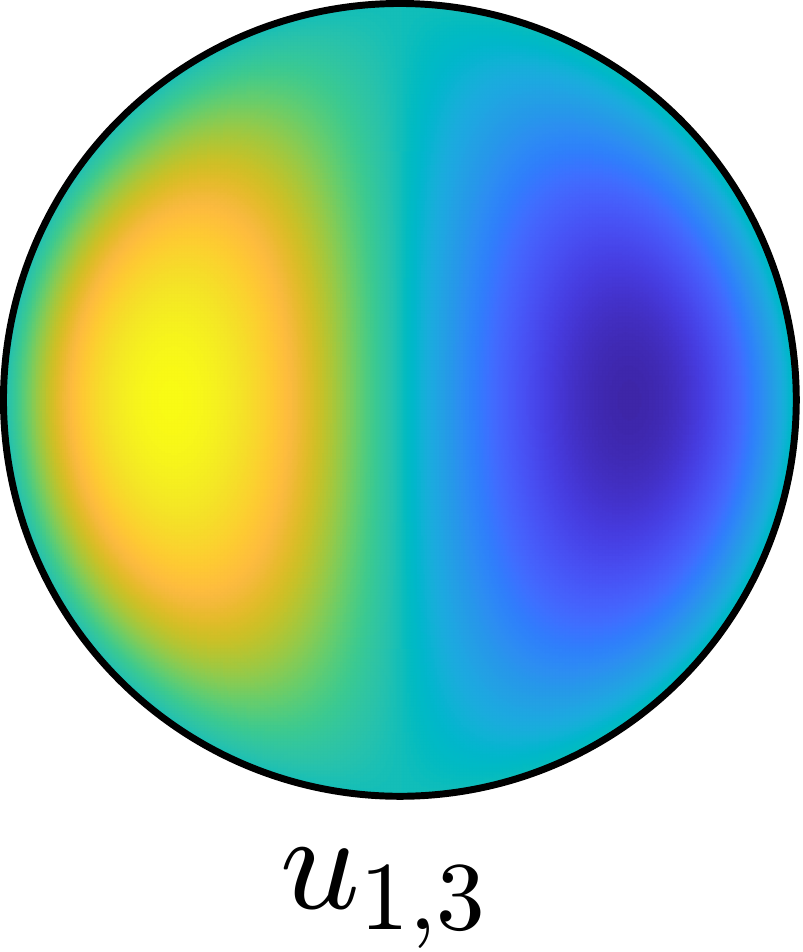}
    \end{subfigure}
    \begin{subfigure}{.16\textwidth}
        \centering
        \includegraphics[width=\linewidth]{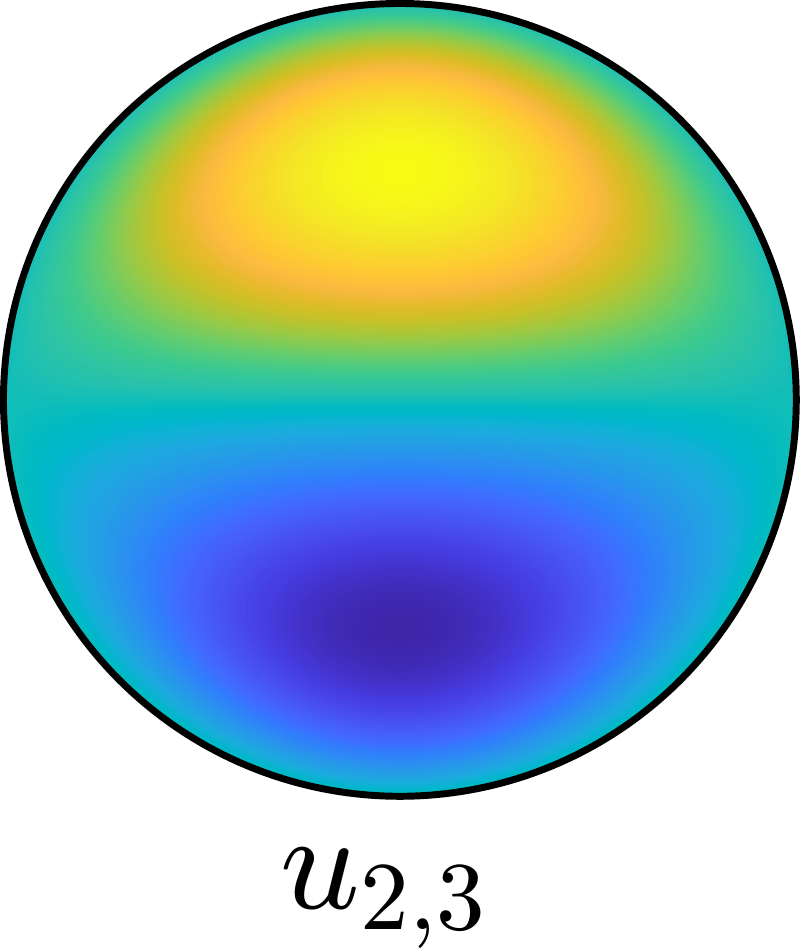}
    \end{subfigure}
    \begin{subfigure}{.16\textwidth}
        \centering
        \includegraphics[width=\linewidth]{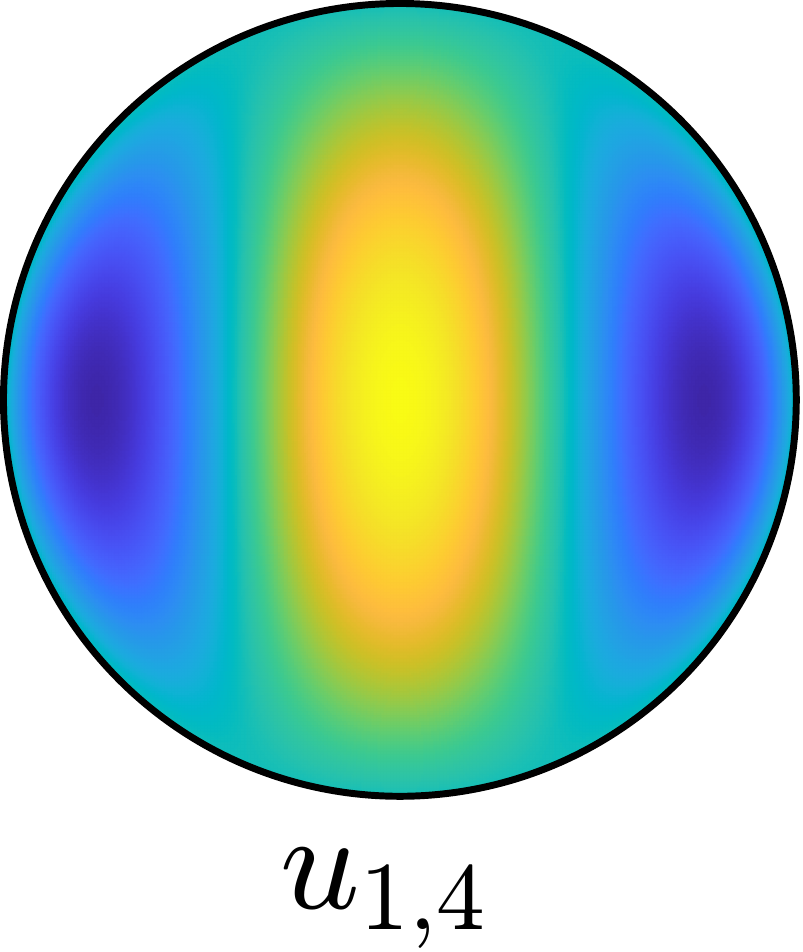}
    \end{subfigure}
    \begin{subfigure}{.16\textwidth}
        \centering
        \includegraphics[width=\linewidth]{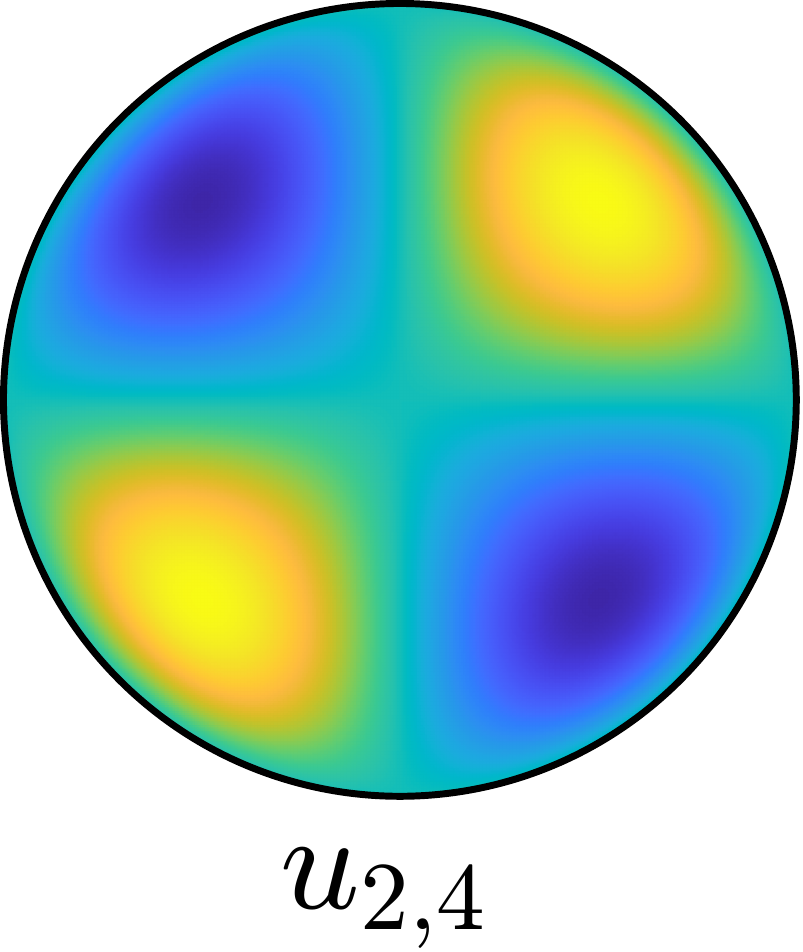}
    \end{subfigure}
    \begin{subfigure}{.16\textwidth}
        \centering
        \includegraphics[width=\linewidth]{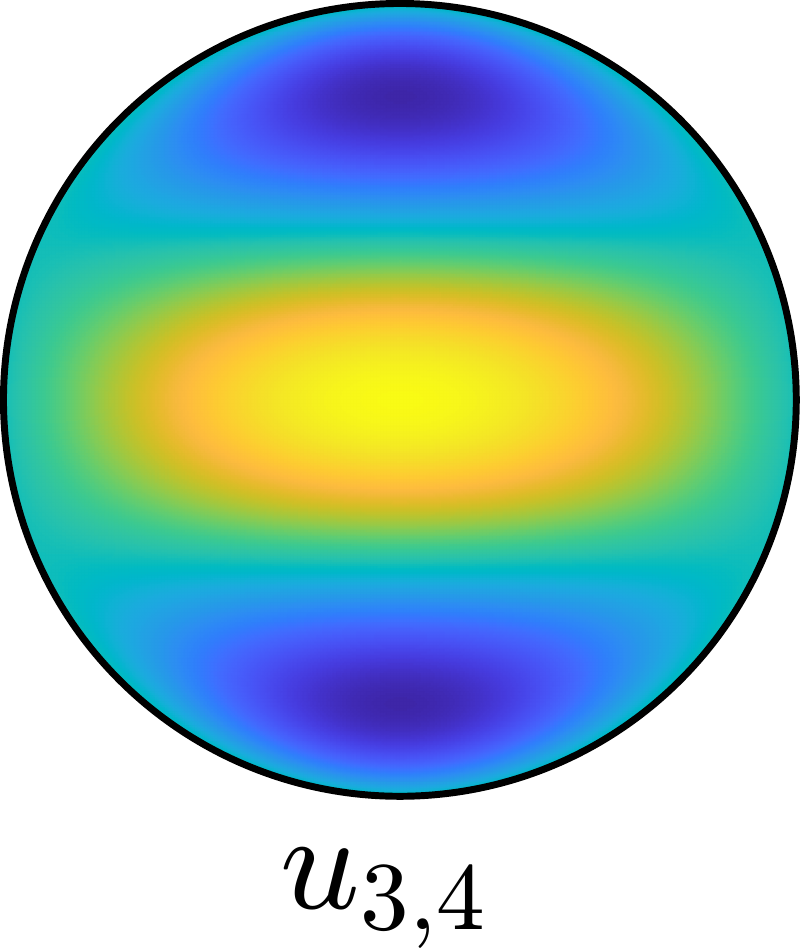}
    \end{subfigure}

    \begin{subfigure}{.16\textwidth}
        \centering
        \includegraphics[width=\linewidth]{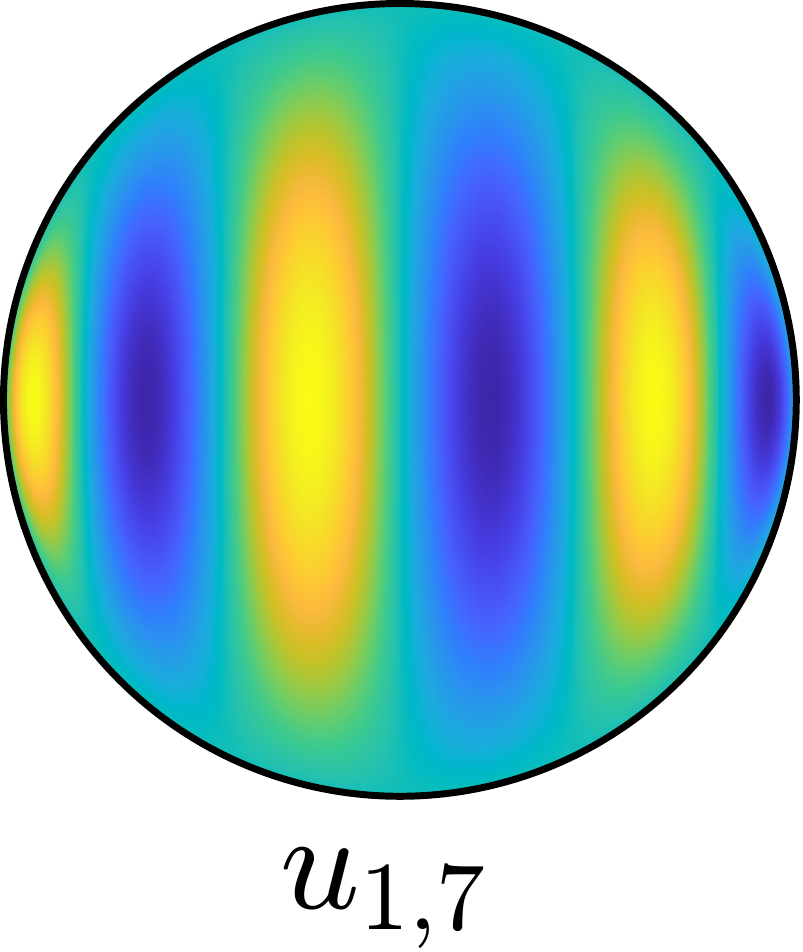}
    \end{subfigure}
    \begin{subfigure}{.16\textwidth}
        \centering
        \includegraphics[width=\linewidth]{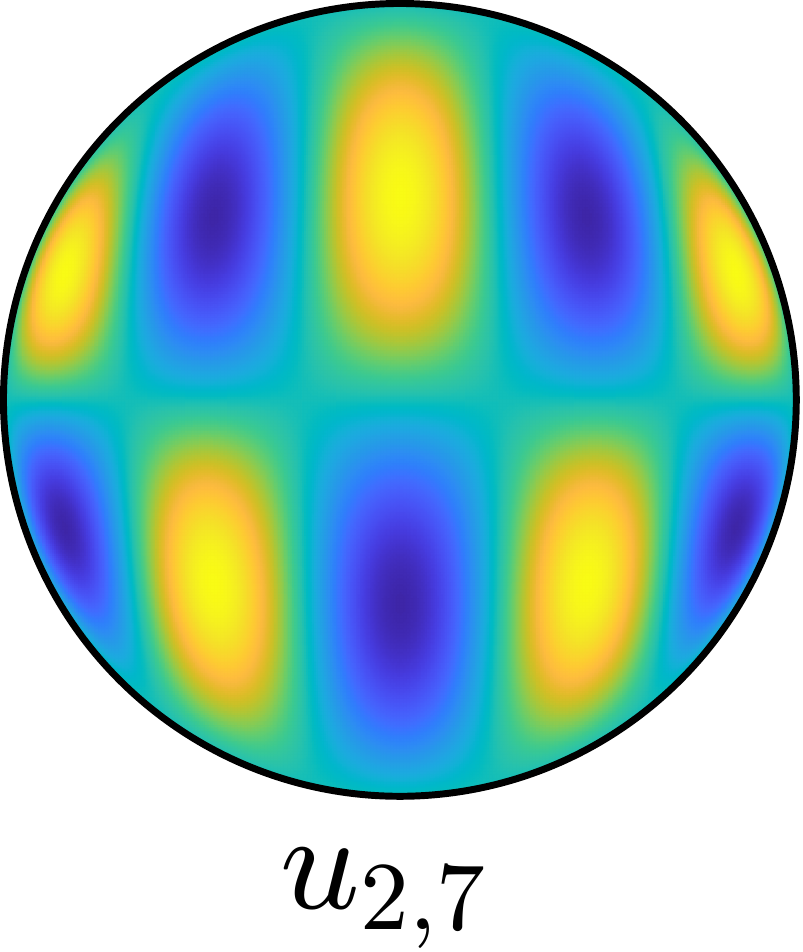}
    \end{subfigure}
    \begin{subfigure}{.16\textwidth}
        \centering
        \includegraphics[width=\linewidth]{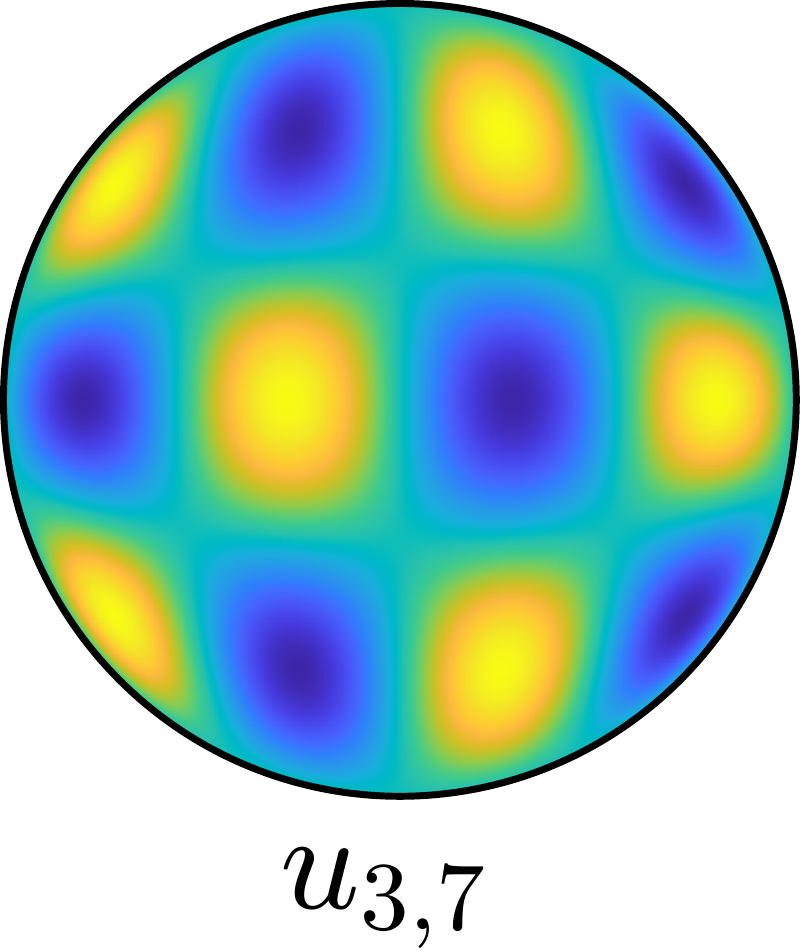}
    \end{subfigure}
    \begin{subfigure}{.16\textwidth}
        \centering
        \includegraphics[width=\linewidth]{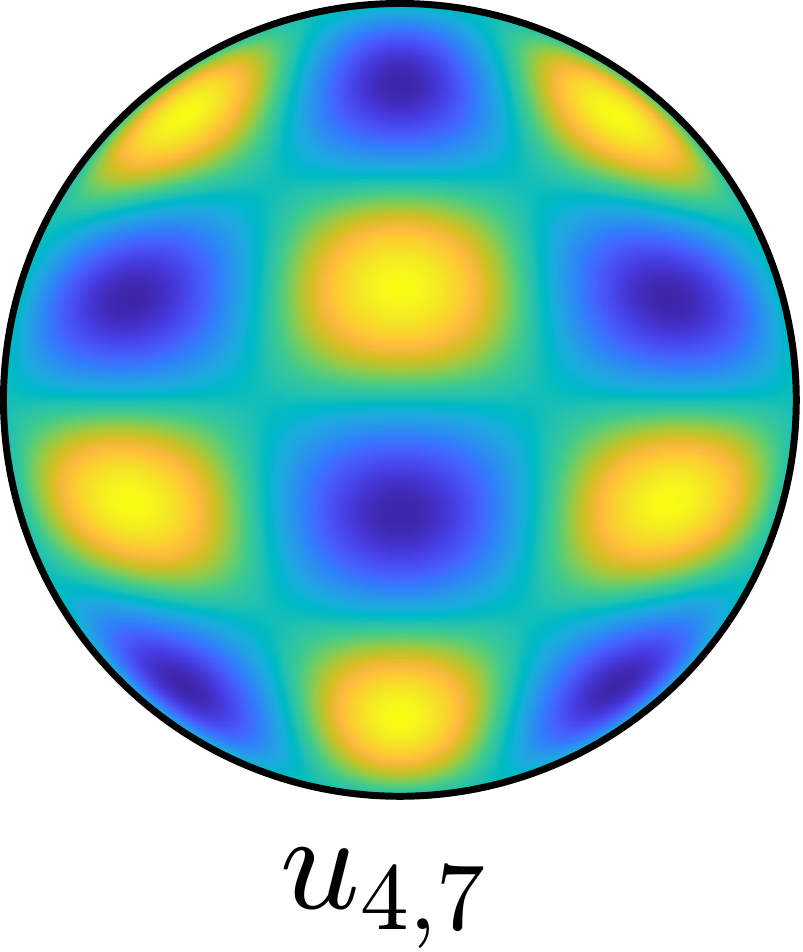}
    \end{subfigure}
    \begin{subfigure}{.16\textwidth}
        \centering
        \includegraphics[width=\linewidth]{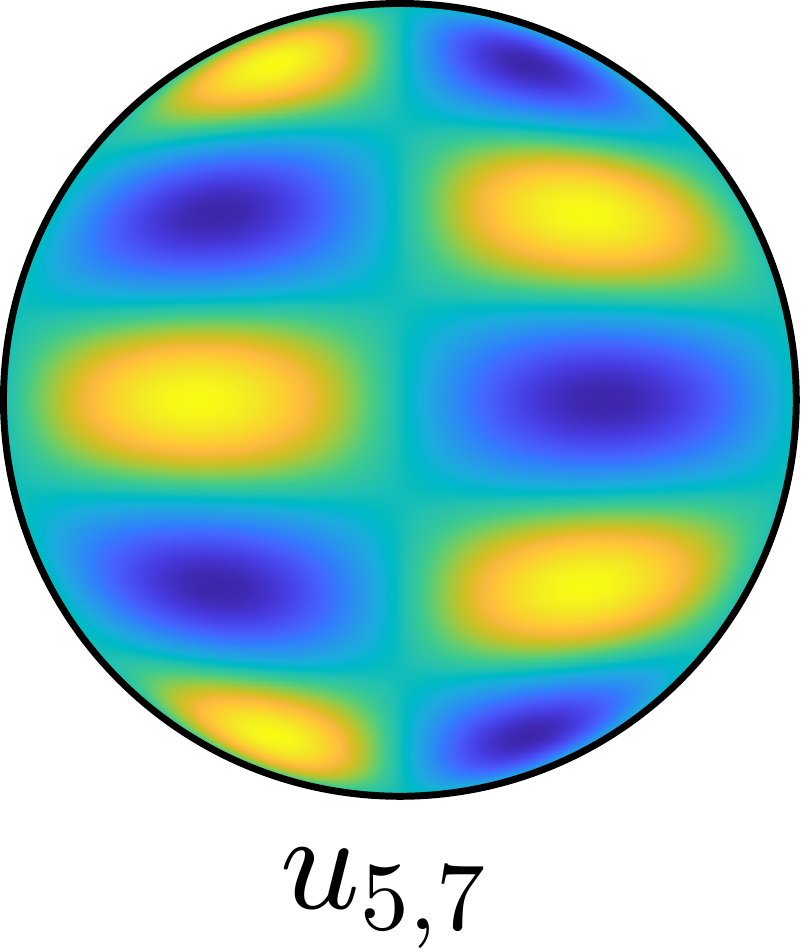}
    \end{subfigure}
    \begin{subfigure}{.16\textwidth}
        \centering
        \includegraphics[width=\linewidth]{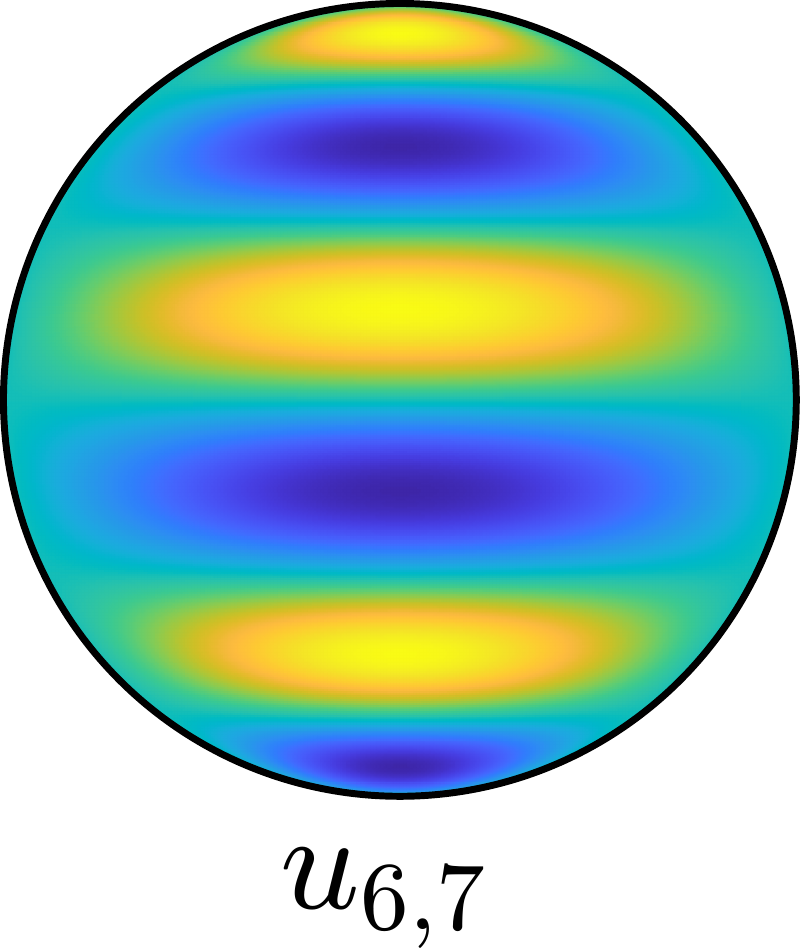}
    \end{subfigure}
    \newline
    
    \vspace{1mm}
    \begin{subfigure}{.3\textwidth}
        \centering
        \includegraphics[width=.9\linewidth]{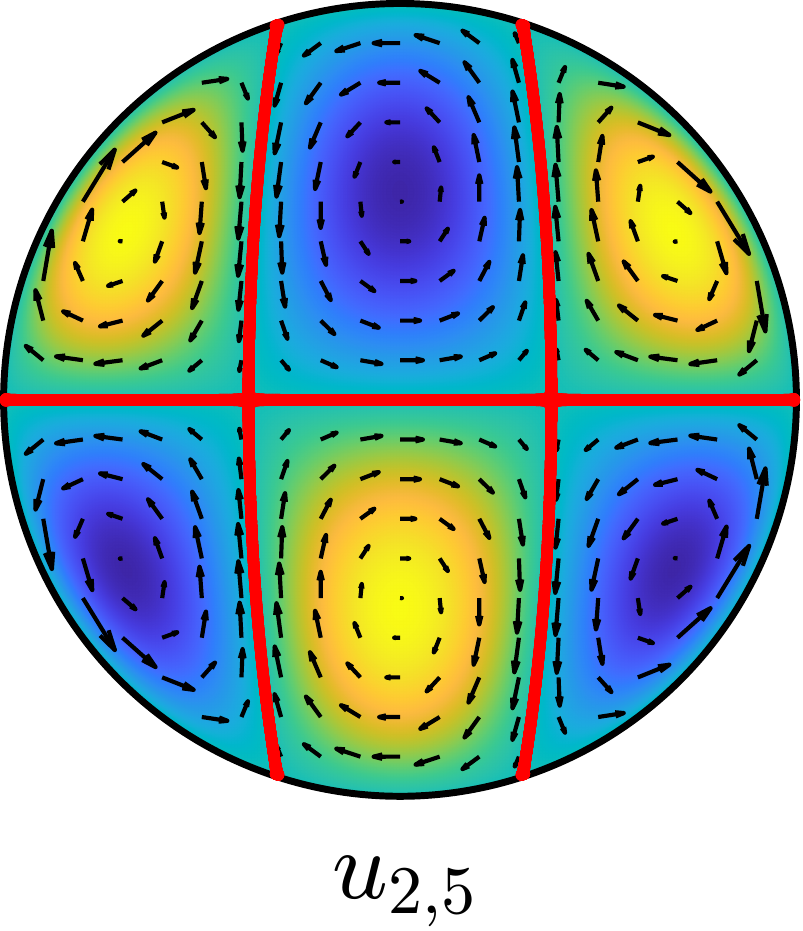}
    \end{subfigure}
    \begin{subfigure}{0.2\textwidth}
        \centering
        \includegraphics[scale = 0.4]{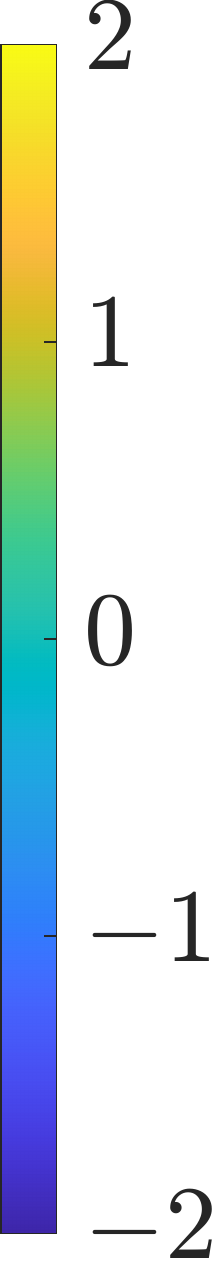}
    \end{subfigure}
    \caption{Plots of
$u = \Delta^{-1}_\omega w$ ($ \Delta_\Omega $ is the Dirichlet Laplacian)
where $w$ is an eigenfunction of the Poincar\'e operator
$P = \partial_{x_2}^2 \Delta_\Omega^{-1}$, see \S \ref{sec:disk}. We also plot the velocity field of $u_{2, 5}$ given by \eqref{eq:vf}. It is clear that the integral curves are simply the level sets of $u_{k, N}$, the zero set (excluding the boundary) is highlighted in red above.} 
    \label{fig:eigfuncs}
\end{figure}

Rewriting \eqref{eq:vf} so that \eqref{eq:internal_waves} reads like an evolution problem, we define 
\begin{equation}\label{eq:P_def}
    P := \partial_{x_2}^2 \Delta_\Omega^{-1}: H^{-1}(\Omega) \to H^{-1} (\Omega), \quad \langle u, w \rangle_{H^{-1}(\Omega)} := \langle \nabla \Delta_\Omega^{-1} u, \nabla \Delta_\Omega^{-1} w \rangle_{L^2(\Omega)}.
\end{equation}
where $\Delta_\Omega$ denotes the Dirichlet Laplacian. Then $w := \Delta u$ satisfies the equation
\begin{equation}\label{eq:evolution_problem}
    (\partial_t^2 + P) w = f \cos \lambda t, \quad w|_{t = 0} = \partial_t w|_{t = 0} = 0, \quad f \in \CIc(\Omega; \R), \quad u = \Delta_\Omega^{-1} w. 
\end{equation}
The operator $P$ is bounded and self-adjoint, see~\cite{Ralston_73}. The solution $w(t)$ can then be written using the functional calculus:
\begin{equation}\label{eq:functional_solution}
    \begin{gathered}
        w(t) = \frac{\cos (t \sqrt{P}) - \cos(t \lambda)}{\lambda^2 - P} f.
    \end{gathered}
\end{equation}
We are interested in its long-time behavior which equivalent to the long time behavior of the solution $u$ to \eqref{eq:internal_waves}. As $t \to \infty$ the functional solution \eqref{eq:functional_solution} becomes singular when the spectral parameter is equal to $\lambda^2$. Therefore studying the long-time behavior is closely related top the spectral properties of $P$ at $\lambda^2$.

Our main result concerns the behavior of $u$ when the underlying classical dynamics is ergodic. The relevant dynamics at forcing frequency $\lambda$ is given by the chess billiard map, which is a $\lambda$-dependent family of circle diffeomorphism $b(\bullet, \lambda) : \partial \Omega \to \partial \Omega$ (see \eqref{eq:b_def} for details and \S\ref{sec:geometric_assumptions} for motivation). To every $b(\bullet, \lambda)$, we may assign a rotation number $\mathbf r(\lambda) \in [0, 1]$, which measures the average rotation of $b(\bullet, \lambda)$, see \eqref{eq:rot_num_def} for the precise definition. When the rotation number $\mathbf r(\lambda)$ is irrational, the map $b(\bullet, \lambda)$ is ergodic, which is the setting we will focus on.

Finally, we need a natural geometric assumption on the domain $\Omega$, called $\lambda$-simplicity \cite{DWZ}, see Definition \ref{lambda_simple_def}. For now, we emphasize that all strictly convex domains satisfy this assumption for all $\lambda \in (0, 1)$.

\begin{theorem}\label{thm:ergodic}
Let $\Omega \subset \R^2$ be open, bounded, and simply connected. Assume that $\lambda \in (0, 1)$ is such that $\Omega$ is $\lambda$-simple. 
\begin{enumerate}[label = (\alph*)]
    \item If the rotation number $\mathbf r (\lambda)$ is irrational, then $\lambda^2$ is not in the pure point spectrum of $P$, i.e. $(P - \lambda^2): H^{-1}(\Omega) \to H^{-1}(\Omega)$ is injective. 

    \item If the rotation number $\mathbf r(\lambda )$ is Diophantine (see Definition \ref{def:diophantine}), the solution $u(t)$ to \eqref{eq:internal_waves} remains bounded in energy space for all times, i.e. there exists a constant $C > 0$ such that
    \begin{equation}
        \|u(t)\|_{H^1_0(\Omega)} \le C
    \end{equation}
    for all $t \in \R$.

    \item If the rotation number  $\mathbf r (\gl )$ is Diophantine, the spectral measure $\mu _f $ of $f$ satisfies, for all $N\in \N$, 
    \[ \mu _f \left([\gl^2 -\epsilon,\gl^2 + \epsilon] \right)= O\left(\epsilon^N\right) \]
\end{enumerate}  
\end{theorem}
Parts (b) and (c) of Theorem \ref{thm:ergodic} were first proved by 
the second named author \cite{Zhenhao}
(Theorem 2 and Theorem 1 respectively). This was done by studying a boundary reduced 0-th order pseudo-differential operator on the circle, found in \cite{DWZ}. The assumptions in (a) are weaker than the corresponding assumption of \cite[Lemma~6.1]{Zhenhao},  which assumes the eigenfunction to be smooth. It was already proven in \cite[Theorem 10]{arnold}, see also \cite{john}.

In this paper, we give a simpler proof by directly constructing an inverse to a related eigenvalue problem (see~\eqref{eq:eig_problem}). We also describe the spectrum in two simple cases: rectangular and elliptic domains;
the latter is related to the recent work of the first named author and Vidal
\cite{C-V} which studied the case rotating 
fluids in ellipsoids.

The spectral measure of $P$ can be numerically approximated using the methods developed in \cite{Colbrook21a,Colbrook21b,Colbrook22}. Theorem \ref{thm:ergodic} still holds for the square even though the boundary is not smooth. In that case, the rotation number can be computed as an explicit smooth function of $\lambda$, and so the spectral bound can be seen numerically in Figure \ref{fig:spectral_bound}.

Two examples where we can explicitly compute the spectrum of $P$ are given
by $\Omega = [0, 1] \times [0, 1]$ and $\Omega = \mathbb D$,  the unit disk. For the square, the eigenfunctions are simply the Fourier modes. For the circle, we will give an explicit complete basis of $H^1_0$ consisting of solutions to the eigenvalue problem \eqref{eq:eig_problem}, see \eqref{basis} and Figure \ref{fig:eigfuncs}. Then by changing the coordinates, we have spectral information about elliptic and rectangular domains.

\begin{figure}
    \centering
    \includegraphics[width = \textwidth]{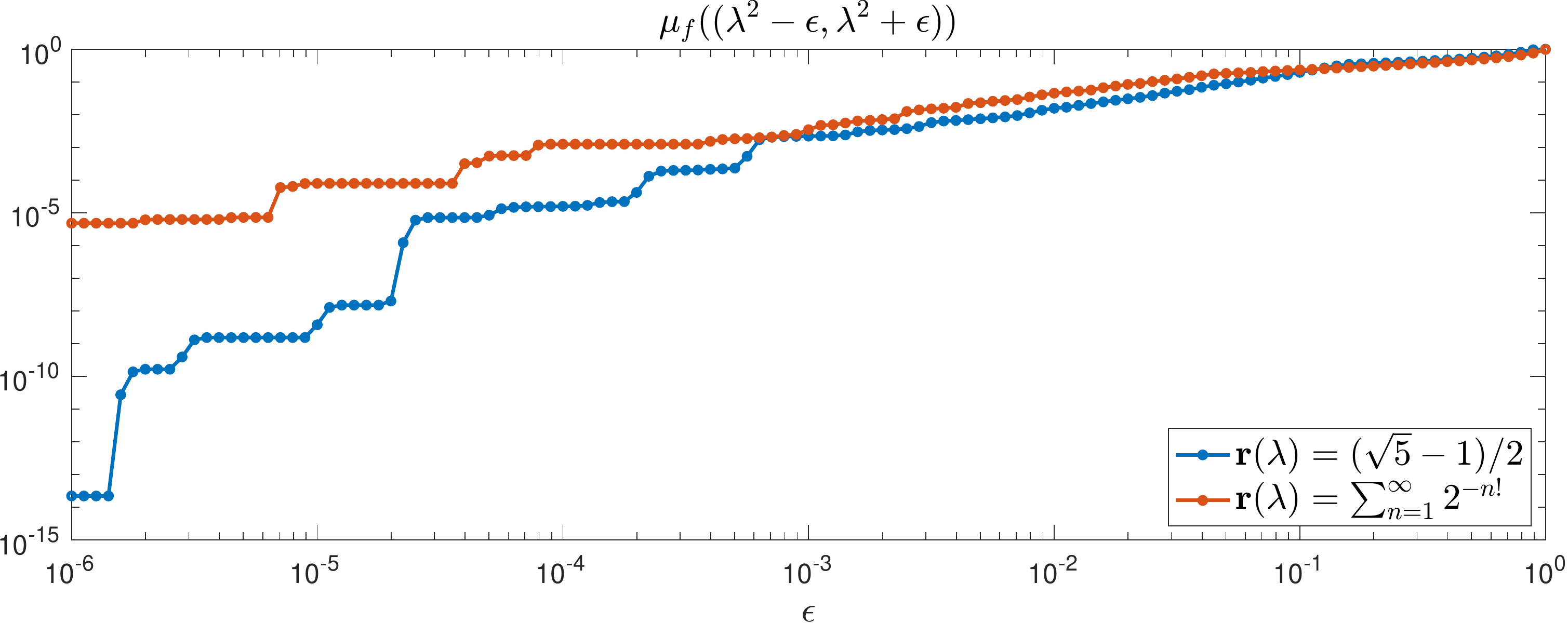}
    \caption{Numerical evidence for the relevance of the Diophantine assumption in part (c) of Theorem \ref{thm:ergodic}. Here $\Omega = [0, 1] \times [0, 1]$. The spectral measure of a $2\epsilon$ interval centered at $\lambda^2$ with $\mathbf r(\lambda)$ Diophantine decays much faster than if $\mathbf r(\lambda)$ is irrational but not Diophantine. It is clear that $\sum_{n = 1}^\infty 2^{-n!}$, which is a base 2 version of Liouville's constant, cannot satisfy Definition \ref{def:diophantine} for any choice of constants $c$ or $\beta$. The data presented is courtesy of Matthew Colbrook.}
    \label{fig:spectral_bound}
\end{figure}

\begin{theorem} \label{thm:ellipse}
If $\Omega$ is an ellipse or if $\Omega$ is a rectangle with sides parallel to the coordinate axis, the spectrum of $P$ is pure point dense in $[0,1]$ and the eigenvalues are exactly the values of $\lambda$ for which the rotation number $\mathbf r(\lambda) $ is rational. 
\end{theorem}
For the theorem above, we emphasize that the ellipses we consider can be any nondegenerate linear transformation of a circle. This includes tilted ellipses with major and minor axis not necessarily parallel to the coordinate axis. On the other hand, the theorem fails for tilted rectangle; it is shown in \cite{Ralston_73} that the spectrum contains an absolutely continuous part.

\subsection{Related works}
Analysis of (\ref{eq:internal_waves}) goes back to \cite{sobolev}. The spectral properties of the operator $P$ defined in (\ref{eq:P_def}) were studied in \cite{Aleksandryan_60} and \cite{Ralston_73}. The study of internal waves has motivated the mathematical analysis of 0-th order self-adjoint pseudodifferential operators. 
Such operators on closed surfaces in the presence of attractors were studied in~\cite{CdV-b, CdV-c} and then in~\cite{Dyatlov_Zworski_19}. The viscosity limits of these 0-th order pseudodifferential operators were recently studied by Galkowski--Zworski \cite{Galkowski_Zworksi_22} and Wang \cite{Wang_22}. Spectral properties of 0-th order pseudodifferential operators on the circle were studied by Zhongkai Tao~\cite{Tao_19} who produced examples of embedded eigenvalues.

 Then for the case of 2D planar domains, \cite{DWZ} showed that when the underlying dynamics has hyperbolic attractors, there can be high concentration of the fluid velocity near these attractors. This phenomenon was predicted in the physics literature by Maas--Lam \cite{maas_lam_95} in 1995, and has since been experimentally observed by Maas et al. \cite{maas_observation_97}, Hazewinkel et al. \cite{Hazewinkel_2010}, and Brouzet \cite{brouzet_16}. We also consider 2D planar domains in our work, but the underlying dynamical assumption is different, thus leading to different conclusions.

\section{Preliminaries}
In this section, we establish the necessary geometric assumptions so that the underlying classical dynamics governing the system can be reduced to the the boundary of~$\Omega$. We then provide an outline of the necessary results from one-dimensional dynamics required to analyze our internal waves model. 

\subsection{Geometric assumptions}\label{sec:geometric_assumptions}
From \eqref{eq:functional_solution}, it is clear that we need to understand the spectral properties of $P$ at $\lambda^2$. We consider the eigenvalue problem
\begin{equation}\label{eq:eig_problem}
    P(\lambda) u = 0 \quad \text{where} \quad P(\lambda) := (P - \lambda^2) \Delta = (1 - \lambda^2) \partial_{x_2}^2 - \lambda^2 \partial_{x_1}^2.
\end{equation}
Clearly, $P(\lambda): H^1_0(\Omega) \to H^{-1}(\Omega)$ is invertible if and only if $\lambda^2$ is not in the spectrum of the operator $P$ defined in \eqref{eq:evolution_problem}. The problem has no nontrivial solutions if $\lambda > 1$, because then the operator $P(\gl)$ is elliptic and the result follows from the maximum principle. There is also no solution for $\gl =0$ or $\gl = 1$ by direct inspection. It is also known that the spectrum of $P$ is precisely $[0, 1]$, see \cite{Ralston_73}.

The advantage of working with $P(\lambda)$ is that it is simply a $(1 + 1)$-dimensional wave operator, and the symbol is given by the quadratic form $-(1 - \lambda^2) \xi_2^2 + \lambda^2 \xi_1^2$. The relevant classical dynamics here is given by the Hamiltonian flow of the symbol. The dual of this quadratic form can be factorized as 
\begin{equation}\label{eq:dual_factor}
    -\frac{x_1^2}{\lambda^2} + \frac{x_2^2}{1 - \lambda^2} = \ell^+(x, \lambda) \ell^-(x, \lambda), \quad \ell^\pm(x, \lambda) := \pm \frac{x_1}{\lambda} + \frac{x_2}{\sqrt{1 - \lambda^2}}.
\end{equation}
In particular, the integral curves of the Hamiltonian flow of the symbol projected onto~$\Omega$ are precisely the level sets of $\ell^\pm(x, \lambda)$.

\begin{definition}\label{lambda_simple_def}
    Let $\lambda \in (0, 1)$. Then $\Omega$ is called \textup{$\lambda$-simple} if each $\ell^\pm(\bullet, \lambda): \partial \Omega \to \R$ has exactly two distinct critical points which are both non-degenerate.
    See Figure~\ref{fig:defs}.
\end{definition}
\begin{figure}
    \centering
    \includegraphics[scale = .25]{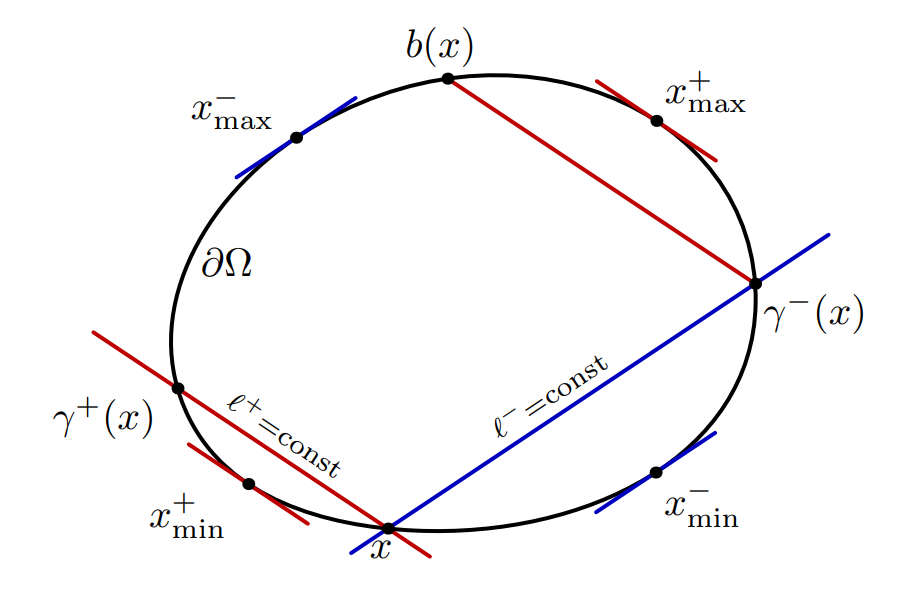}
    \caption{Definition of the involutions $\gamma^\pm$ as well as the chess billiard map $b(x)$. Also illustrates the definition of $\lambda$-simplicity, since $\ell^\pm$ is nondegenerate at the critical points $x_{\mathrm{min}}^\pm$ and $x_{\mathrm{max}}^\pm$. The diagram is from \cite{DWZ}, which considered the same dynamical system.}
    \label{fig:defs}
\end{figure}

Under the $\lambda$-simple condition, the dynamics can be reduced to the boundary. In particular, there exist unique smooth orientation-reversing involutions $\gamma^\pm(\bullet, \lambda): \partial \Omega \to \partial \Omega$ that satisfy
\begin{equation}\label{eq:gamma_def}
    \ell^\pm(x) = \ell^\pm(\gamma^\pm(x)).
\end{equation}
Composing $\gamma^+$ and $\gamma^-$, we obtain an orientation preserving diffeomorphism of the boundary $\partial \Omega$. This is known as the \textit{chess billiard map} $b(\bullet, \lambda): \partial \Omega \to \partial \Omega$, defined by
\begin{equation}\label{eq:b_def}
    b:= \gamma^+ \circ \gamma^-.
\end{equation}
See Figure \ref{fig:defs}. We will often suppress the dependence on $\lambda$ in the notation when there is no ambiguity. 

\subsection{One-dimensional dynamics}
Let $\mathbf x: \mathbb S^1 = \R/\Z \to \partial \Omega$ be a positively oriented parametrization on~$\partial \Omega$.
This gives rise to a covering map $\tilde \pi: \R \to \partial \Omega$ given by $\tilde \pi(x) = \mathbf x(x \bmod 1)$. Let $\mathbf b(\bullet, \lambda) : \R \to \R$ be a lift of $b(\bullet, \lambda)$, i.e. $\mathbf b(\bullet, \lambda)$ satisfies
\[\pi(\mathbf b(x, \lambda)) = b(\pi(x), \lambda)\]
for all $x \in \R$. Fix an initial point $x_0 \in \R$. The \textit{rotation number} of $b( \bullet, \lambda)$ is then defined as
\begin{equation}\label{eq:rot_num_def}
    \mathbf r(\lambda) := \lim_{k \to \infty} \frac{\mathbf b^k(x_0, \lambda) - x_0}{k} \in \R/\Z.
\end{equation}
The rotation number is simply the averaged rotation along an orbit, and it is a well-defined quantity:
\begin{lemma}\label{lem:rot_num}
    The rotation number $\mathbf r(\lambda)$ as defined by the limit (\ref{eq:rot_num_def}) exists. Furthermore, it is independent of the choice of initial point $x_0$, the parametrization $\mathbf x$, and the lift $\mathbf b$. 
\end{lemma}
See e.g.~\cite{Walsh} for a proof of Lemma \ref{lem:rot_num}. We refer the reader to \cite[Chapter~1]{Melo_Strien} for a thorough discussion of circle homeomorphisms.

We wish to understand to what extent a circle diffeomorphism $b$ can be conjugated to a rotation. This is clearly impossible for rational rotations, and it was proved by Denjoy in \cite{denjoy-or} that when the rotation number is irrational, the circle diffeomorphism is topologically conjugate to a rotation. We will need more regularity than just topologically conjugate, which
is known when the rotation number satisfies the following Diophantine condition:
\begin{definition}\label{def:diophantine}
    A number $\alpha \in \R$ is \textup{Diophantine} if there exists constants $c, \, \beta > 0$ so that 
    \begin{equation*}
        \left|\alpha - \frac{p}{q} \right| \ge \frac{c}{q^{2 + \beta}}
    \end{equation*}
    for all $p \in \Z$ and $q \in \N$. 
\end{definition}

Roughly speaking, this is a ``sufficiently irrational" condition; $\alpha$ is Diophantine if it is sufficiently far away from all rational numbers. Diophantine numbers exist and form a full measure set. Indeed, take some $\beta > 0$ and a small $c > 0$, and consider the set of numbers in (0,1) that are Diophantine with respect to $\beta$ and $c$:
\[E = \left \{\alpha \in (0, 1) : \left|\alpha - \frac{p}{q} \right| \ge \frac{c}{q^{2 + \beta}} \, \text{ for all $p \in \Z$ and $q \in \N$} \right\}.\]
We can think of $E$ as the set remaining after stripping away small neighborhoods of rational numbers, so the measure of $E$ is at least 
\[|E| \ge 1 - \sum_{q = 1}^\infty 2q \cdot \frac{c}{q^{2 + \beta}}.\]
The right hand side becomes arbitrarily close to $1$ as $c \to 0$.

If the rotation number is irrational, then the topological conjugation can be upgraded to smooth conjugation to a rotation. This upgrade is due to Herman \cite{Herman_79} and Yoccoz \cite{yaccoz}. We collect the conjugation results from \cite{denjoy-or, Herman_79, yaccoz} in the following proposition. 

\begin{proposition}\label{prop:conjugation}
    Let $b: \mathbb S^1 \to \mathbb S^1$ be a smooth circle diffeomorphism with rotation number $\alpha$. 
    \begin{enumerate}[label = (\alph*)]
        \item if $\alpha$ is irrational, then there exists a homeomorphism $\psi: \mathbb S^1 \to \mathbb S^1$ such that $(\psi \circ b \circ \psi^{-1})(\theta) = \theta + \alpha$.

        \item if $\alpha$ is Diophantine, then there exists a smooth diffeomorphism $\psi: \mathbb S^1 \to \mathbb S^1$ such that $(\psi \circ b \circ \psi^{-1})(\theta) = \theta + \alpha$.
    \end{enumerate}
\end{proposition}
A standard consequence of the topological conjugacy result is the unique ergodicity of circle diffeomorphisms with irrational rotation number.
\begin{corollary}\label{cor:ergodic}
    Let $b: \mathbb S^1 \to \mathbb S^1$ be a smooth circle diffeomorphism with irrational rotation number. Then for every $f \in C^0(\mathbb S^1)$, there exists a constant function $c(f)$ such that
    \[\frac{1}{N} \sum_{n = 0}^{N - 1} (b^*)^n f \to c(f) \quad \text{in $C^0(\mathbb S^1)$ as $N \to \infty$}.\]
    Consequently, if $b^* f = f$, then $f$ must be constant. 
\end{corollary}
\begin{proof}
    By Proposition \ref{prop:conjugation}, it suffices to prove the theorem in the case that $b(\theta) = \theta + \alpha$ for some irrational $\alpha \in (0, 1)$. Furthermore, it suffices to prove the statement for a dense subset of $C^0(\mathbb S^1)$, so we simply check the statement for trigonometric polynomials. Indeed, it is trivial for constant functions, and for $k \in \Z \setminus \{0\}$, 
    \[\frac{1}{N} \sum_{n = 0}^{N - 1} e^{2 \pi ik(\theta + n \alpha)} = \frac{1}{N} e^{2 \pi ik \theta} \frac{1 - e^{2 \pi i k \alpha N}}{1 - e^{2 \pi i k \alpha}} \to 0\]
    uniformly in $\theta$ since $1 - e^{2 \pi ik \alpha} \neq 0$ by the irrationality of $\alpha$. This proves the statement for all trigonometric polynomials as desired. 
\end{proof}


Eventually, we will need to solve cohomological equations of the form
\begin{equation}\label{eq:basic_cohomological}
    v(\theta)-v(\theta + \alpha)=g(\theta)
\end{equation}
given $g \in \CIc(\mathbb S^1)$ with vanishing integral. If $\alpha$ satisfies the Diophantine condition, then we can have explicit high frequency control of $v$. In particular, if $g$ is smooth, then we have smooth solutions. 

\begin{lemma}\label{lem:cohom}
    Let $g \in C^\infty (\mathbb S^1)$ with $\int_{\mathbb S^1} g = 0$, and assume that $\alpha$ is Diophantine. Then there exists $v \in C^\infty(\mathbb S^1)$ that solves \eqref{eq:basic_cohomological} uniquely modulo constant functions.
\end{lemma}
\begin{proof}
    Let $c, \beta > 0$ be the Diophantine constants associated with $\alpha$ in Definition \ref{def:diophantine}. Taking the Fourier series of both sides of (\ref{eq:basic_cohomological}), 
    \[\hat g(k) = (1 - e^{2 \pi i k \alpha}) \hat v(k).\]
    For all $k \neq 0$,  
    \[\left|\frac{1}{1 - e^{2 \pi i k \alpha}} \right| \le c^{-1} |k|^{1 + \beta}.\]
    Since $\hat g(0) = 0$, $v$ can be solved uniquely up to choice of $\hat v(0)$, and we have the estimate
    \[\|v\|_{H^s} \le C\|g\|_{H^{s + \beta + 1}}\]
    for every $s > 0$. Therefore $v$ is smooth and unique modulo constant functions. 
\end{proof}

\section{Internal waves in an ergodic setting}

In this section we prove Theorem \ref{thm:ergodic}. The strategy is to construct a right inverse to $P(\lambda)$ on smooth functions for $\lambda$ such that $\mathbf r(\lambda)$ is Diophantine (Definition \ref{def:diophantine}). Define coordinates
\begin{equation}\label{eq:y_coord}
    y_\pm = \frac{1}{2} \ell^\pm (x, \lambda)
\end{equation}
where $\ell^\pm$ are defined in \eqref{eq:dual_factor}. In these coordinates, we have
\[P(\lambda) = \frac{\partial^2}{\partial y_+ \partial y_-}.\]
In this form, it is clear that solutions on the interior of $\Omega$ to the eigenvalue problem \eqref{eq:eig_problem} can be found by integration up to boundary conditions, which effectively reduces the problem to the boundary. 

\subsection{Injectivity of the eigenvalue problem}
We first show that solutions to the eigenvalue problem take a very convenient form. 
\begin{lemma}\label{lem:decomp}
    Let $\Omega$ be $\lambda$-simple and suppose $u \in H^1_0(\Omega)$ is a solution to the eigenvalue problem \eqref{eq:eig_problem}. Then using the coordinates \eqref{eq:y_coord}, $u$ can be decomposed as
    \[u(y_+, y_-) = u_+(y_+) + u_-(y_-)\]
    where $u_\pm \in H^1(\Omega)$ are functions on $\Omega$ that depend only on $y_\pm$ respectively. In fact, $u_\pm \in C^0(\overline \Omega)$. 
\end{lemma}
\begin{proof}
    Let $u'_+ := \partial_{y_+} u$. Observe that $\partial_{y_-} u'_+ = 0$, so $u'_+$ is a function of $y_+$ only. The eigenvalue problem \eqref{eq:eig_problem} is invariant under shifting of the domain, so we may assume without the loss of generality that the two nondegenerate critical points of $\ell^+$ are $y_+ = 0$ and $y_+ = a$. Note that $u'_+ \in L^2(\Omega)$, so the nondegeneracy of the critical points imply that
        \[\int_0^1 \sqrt{y_+(a - y_+)} |u'_+(y_+)|^2\, dy_+ < \infty.\]
    By Cauchy-Schwarz, we then have
    \[\|u'_+\|_{L^1([0, a])} \le \|(y_+(1 - y_+))^{1/4} u'_+\|_{L^2([0, a])} \|(y_+(1 - y_+))^{-1/4}\|_{L^2([0, a])} < \infty\]
    Put 
    \[\tilde u_+(y_+) := \int_{0}^{y_+} u'_+(s)\, ds\]
    By a slight abuse of notation, we can view $\tilde u_+$ as an element of $C^0(\Omega)$ that depends only on $y_+$. Then we see that $u - \tilde u_+ \in H^1$ is a function that depends only on $y_-$. 

    Running the same argument with $y_-$ instead of $y_+$, we also have $\tilde u_- \in C^0(\Omega)$ depending only on $y_-$ such that $u - \tilde u_-$ depends only on $y_+$. Modifying $\tilde u_\pm$ by constants, there must exist $u_\pm \in C^0(\Omega)$ that depends only on $y_\pm$ such that $u_+ + u_- = 0$. 
\end{proof}

\begin{proof}[Proof of Theorem \ref{thm:ergodic}(a)]
    We prove that the operator $P(\lambda)$ has trivial kernel. Suppose $u$ is a solution the eigenvalue problem \eqref{eq:eig_problem}. Then we have the decomposition $u(y_+, y_-) = u_+(y_+) + u_-(y_-)$ from Lemma \ref{lem:decomp}. Define the boundary traces
    \begin{equation}\label{eq:boundary_reduction}
        U_\pm := u_\pm|_{\partial \Omega} \in C^0(\Omega).
    \end{equation}
    Observe that $U_+ + U_- = 0$ and $U_\pm \circ \gamma_\pm = U_\pm$, so it follows that $U_\pm$ are invariant under pullback by the chess-billiard map:
    \begin{equation}\label{eq:b_invariance}
        U_\pm = b^* U_\pm.
    \end{equation}
    If the rotation number $\rho(\lambda)$ of $b( \bullet, \lambda)$ is irrational, it follows from \eqref{eq:b_invariance} and Corollary \ref{cor:ergodic} that the functions $U_\pm$ must be constant. The solution $u$ can then be recovered from the boundary data $U_\pm$, and we clearly have $u = 0$. 
\end{proof}
   
\subsection{Energy boundedness}
We first construct the right inverse on smooth functions to $P(\lambda)$ when the rotation number defined in \eqref{eq:rot_num_def} is Diophantine.  
\begin{proposition}\label{prop:right_inverse}
    Let $\Omega$ be $\lambda$-simple and assume that the rotation number $\mathbf r(\lambda)$ is Diophantine. Then there exists $R(\lambda): C^\infty(\overline \Omega) \to C^\infty(\overline \Omega) \cap H^1_0(\Omega)$ such that $P(\lambda) R(\lambda) = \Id$. 
\end{proposition}
\begin{proof}
    1. Using the coordinates $y_\pm$ defined in \eqref{eq:y_coord}, we wish to solve
    \begin{equation}\label{eq:u}
        \frac{\partial^2 u}{\partial y_+ \partial y_-} = f, \quad f \in C^\infty(\overline \Omega)
    \end{equation}
    for $u \in H_0^1(\Omega)$. Let $\tilde f \in \CIc(\R^2)$ be an extension of $f$, i.e. $\tilde f|_{\Omega} = f$. Define
    \[u_0(y_+, y_-) := \int_{-\infty}^{y_+} \int_{-\infty}^{y_-} \tilde f(\eta_+, \eta_-) \, d\eta_- d\eta_+, \quad \text{and} \quad U_0 := u_0|_{\partial \Omega} \in C^\infty(\partial \Omega).\]
    Therefore, it suffices to solve
    \begin{equation}\label{eq:v}
        \frac{\partial^2 v}{\partial y_+ \partial y_-} = 0, \quad v|_{\partial \Omega} = U_0
    \end{equation}
    for $v \in C^\infty(\Omega)$ since $u = u_0 - v$ would then be the solution to \eqref{eq:u}. 

    \noindent
    2. By Proposition \ref{prop:conjugation}, we may choose smooth coordinates $\theta$ on the boundary $\partial \Omega$ so that $b(\theta) = \theta + \mathbf r(\lambda)$. We claim that
    \begin{equation}\label{eq:zero_avg}
        \int_{\partial \Omega} U_0(\gamma^\pm(\theta)) - U_0(\theta)\, d \theta = 0.
    \end{equation}

Recall by \eqref{eq:b_def}, we have $b\circ\gamma^+=\gamma^+\circ b^{-1}$. In the coordinate
$\theta$ this gives $\gamma^+(\theta)+ \mathbf r(\lambda) =\gamma^+(\theta - \mathbf r(\lambda))$. Differentiating in~$\theta$, we get
\[\partial_\theta \gamma^+(\theta)=\partial_\theta \gamma^+(\theta- \mathbf r(\lambda)).\] 
Since $\mathbf r(\lambda)$ is irrational and $\gamma^+$ is an orientation reversing involution, we see that
$\partial_\theta\gamma^+(\theta)=-1$. Similarly, $\partial_\theta \gamma^-(\theta) = -1$, hence \eqref{eq:zero_avg} follows.


    \noindent
    3. Therefore, it follows from Lemma \ref{lem:cohom} that there exist unique $V_\pm \in C^\infty(\partial \Omega)$ such that
    \begin{equation}\label{eq:V_cohom}
        (b^{\pm 1})^* V_\pm - V_\pm = (\gamma^\mp)^* U_0 - U_0 \quad \text{and} \quad \int_{\partial \Omega} V_\pm(\theta)\,d\theta = \frac{1}{2} \int_{\partial \Omega} U_0(\theta) \, d\theta. 
    \end{equation}
    Observe that 
    \begin{align*}
        (V_+ + V_-) - (b^{-1})^*(V_+ + V_-) &= (b^{-1})^*((\gamma^-)^* U_0 - U_0) - ((\gamma^+)^* U_0 - U_0) \\
        &= U_0 - (b^{-1})^* U_0.
    \end{align*}
    Solutions to the cohomological equation are unique up to constants, and it follows from~\eqref{eq:V_cohom} that $\int V_+ + V_- = \int U_0$, so we must have
    \begin{equation}\label{eq:bd_sum}
        V_+ + V_- = U_0.
    \end{equation}
    Furthermore, applying $\gamma^\mp$ to both sides of \eqref{eq:V_cohom}, we see that
    \[((\gamma^\pm)^* V_\pm) - (b^{\pm 1})^* ((\gamma^\pm)^* V_{\pm}) =  U_0 - (\gamma^\mp)^* U_0.\]
    Adding this to \eqref{eq:zero_avg},
    \[[(\gamma^\pm)^* V_\pm - V_\pm] - (b^{\pm 1})^* [(\gamma^\pm)^* V_\pm - V_\pm] = 0.\]
    Again using Lemma \ref{lem:cohom}, we find that
    \begin{equation}\label{eq:gamma_invariance}
        (\gamma^\pm)^* V_\pm = V_\pm.
    \end{equation}
    By the $\lambda$-simplicity assumption, the coordinate functions restricted to boundary, $y_\pm|_{\partial\Omega}$, have nondegenerate critical points. In particular, up to a smooth change of coordinates, we may assume that $y_+ = 0$ is a critical point of $y_+|_{\partial \Omega}$, near which the boundary can be parameterized by $y_-$ and is given by $\{y_+ = y_-^2\}$. We may further assume that the $\gamma^+$-invariance of $V_+$ from \eqref{eq:gamma_invariance} in these coordinates reads $V_+(y_-) = V_+(-y_-)$ near the critical point. Therefore $V_+$ is a smooth function of $y_+ = y_-^2$ near the critical point. Similar analysis holds for all other critical points, so, there exists $v_\pm \in C^\infty(\overline \Omega)$ such that $v_\pm$ depends only on $y_\pm$ and $v_\pm|_{\partial \Omega} = V_{\pm}$. Then $v = v_+ + v_-$ solves \eqref{eq:v} and the boundary conditions are satisfied due to~\eqref{eq:bd_sum}.
\end{proof}

The obstruction to energy boundedness in the functional calculus solution \eqref{eq:functional_solution} is the singularity at $z=\lambda^2$ that appears as $t \to \infty$. This can be cancelled out using the right inverse to $P(\lambda)$ constructed in the previous proposition.  

\begin{proof}[Proof of Theorem \ref{thm:ergodic}(b)]
Let $f \in C^\infty(\overline \Omega; \R)$ and let $R(\lambda)$ be as in Proposition \ref{prop:right_inverse}. Then $g := \Delta_\Omega R(\lambda) f$ lies in $C^\infty(\overline \Omega) \subset H^{-1}(\Omega)$. Then the evolution problem \eqref{eq:evolution_problem} can be rewritten as
\begin{equation*}
    (\partial_t^2 + P) w = (P - \lambda^2) g \cos \lambda t, \quad w|_{t = 0} = \partial_t w|_{t = 0} = 0, \quad u = \Delta_\Omega^{-1} w. 
\end{equation*}
Using the functional calculus solution formula for $w(t)$ given in \eqref{eq:functional_solution}, we have
\begin{equation}\label{eq:functional_solution_var}
    w(t) = (\cos (t \sqrt{P}) - \cos(t \lambda))g.
\end{equation}
Since $\left|\cos (t \sqrt{z}) - \cos(t \lambda) \right| \le 2$ for all $z \in [0, 1]$ and $t \in \R$, it follows from the spectral theorem that $w(t)$ given by \eqref{eq:functional_solution_var} is uniformly bounded in $H^{-1}(\Omega)$ for all $t$. Therefore the solution to the internal waves equation \eqref{eq:internal_waves} given by $u(t) = \Delta_\Omega^{-1}w(t)$ is uniformly bounded in $H^1_0(\Omega)$ for all $t$.  
\end{proof}

\begin{Remark}
    In fact, it is easy to see that smoothness of $f$ is not required. It suffices to have $f \in H^N(\Omega)$ for sufficiently large $N > 0$ depending on the constants in Definition \ref{def:diophantine} of Diophantine numbers. The proof of Lemma \ref{lem:cohom} gives explicit estimates for the regularity of the solution to the cohomological equation. This means that the boundary traces $V_\pm$ from Proposition \ref{prop:right_inverse} lies in $H^{N - k}(\partial \Omega)$ for some $k$ depending on the Diophantine constants. Since the critical points of the coordinate functions $y_\pm$ restricted to the boundary are nondegenerate and $V_\pm$ are invariant under pullback by $\gamma^\pm$ respectively, it follows that $v_\pm \in H^{N/2 - k}(\Omega)$ for a possibly different $k$ depending only on the Diophantine constants.
\end{Remark}

\subsection{Spectral estimate near \texorpdfstring{$\lambda^2$}{}}
We now use the right inverse $R(\lambda)$ in Proposition~\ref{prop:right_inverse} to obtain bounds on the spectral measure near $\lambda$. 
\begin{proof}[Proof of Theorem \ref{thm:ergodic}(c)]
    Let $f \in C^\infty(\overline \Omega) \subset H^{-1}(\Omega)$. The spectral measure $d\mu_f$ satisfies
    \[\int \varphi \,d \mu_f = \langle \varphi(P) f, f \rangle.\]
    Put $f_k = (\Delta_\Omega R(\lambda))^k f$. Note that $f_k \in C^\infty(\overline \Omega)$ and $(P - \lambda^2)^k f_k = f$. Therefore, 
    \[\left |\int_{\lambda^2 - \epsilon}^{\lambda^2 + \epsilon} \, d\mu_f \right| = \left|\int_{\lambda^2 - \epsilon}^{\lambda^2 + \epsilon} (x - \lambda^2)^{2k} \, d\mu_{f_k}(x) \right| \le C_k \epsilon^{2k}\]
    as desired.
\end{proof}

\section{Examples}
In this section we give the explicit computations related to the spectrum of the square and the disk, from which Theorem \ref{thm:ellipse} follows immediately. Note that the square does not satisfy the hypothesis of Theorem \ref{thm:ergodic} since the boundary is not smooth. Nevertheless, the conclusions of Theorem \ref{thm:ergodic} holds, and we verify this directly. 

\subsection{The square}
We consider the square domain $\Omega = [0, 1] \times [0, 1]$. Clearly the Fourier modes provide a basis of eigenfunctions and it is easy to see that the eigenvalues are dense in $[0, 1]$, so Theorem \ref{thm:ellipse} for the square is clear.

The square thus has the advantage that (\ref{eq:internal_waves}) can be solved directly in Fourier series, so we can verify the contents of Theorem \ref{thm:ergodic} for the square domain, despite it having corners. 

The chess billiard flow on the square is the same as the standard billiard flow, so the rotation number function $\mathbf r(\lambda)$ defined in (\ref{eq:rot_num_def}) is smooth and can be written down explicitly. It is given by
\begin{equation*}
    \mathbf r(\lambda) = \frac{\lambda}{\sqrt{1 - \lambda^2} + \lambda}.
\end{equation*}
See \cite{RSI} for the full derivation. We can formally write
\[u(t, x) = \sum_{\mathbf k \in \N^2} \hat u(t, \mathbf k) \sin(\pi k_1 x_1) \sin (\pi k_2 x_2) \]
where $\mathbf k = (k_1, k_2) \in \N^2$. Only in this subsection, we use the hat to denote Fourier transform with respect to the Dirichlet sine basis. If $u(t, x)$ is a solution to (\ref{eq:internal_waves}), the coefficients must satisfy the periodically driven harmonic oscillator equation
\begin{equation}\label{eq:f_coeffs}
    \begin{gathered}
        -\pi^2 (k_1^2 + k_2^2) \partial_t^2 \hat u(t, \mathbf k) + \pi^2 k_2^2 \hat u(t, \mathbf k) = \hat f(\mathbf k) \cos \lambda t, \\
    \text{where} \quad \hat f(\mathbf k) = \int_{[0, 1]^2} f(x) \sin (\pi k_1 x_1) \sin(\pi k_2 x_2)\, dx_1 dx_2
    \end{gathered}
\end{equation}
This has solution 
\begin{equation}\label{eq:square_solution}
    \hat u (t, \mathbf k) = \frac{\hat f (\mathbf k)}{\lambda^2 k_1^2 - (1 - \lambda^2) k_2^2} \left[\cos(\lambda t) - \cos \left( \frac{k_2}{|\mathbf k|} t \right) \right]
\end{equation}
If $\mathbf r(\lambda)$ is Diophantine, then the result of Theorem \ref{thm:ergodic} holds. In fact, in this case, we get that $u(t) \in C^\infty(\Omega)$ uniformly in all the seminorms for all $t \in \R$. Indeed, if $\mathbf r(\lambda)$ is Diophantine, then there exist constants $c,\, \beta > 0$ such that
\[|q \cdot \mathbf r(\lambda) - p| \ge \frac{c}{q^{1 + \beta}}\]
for any $p \in \Z$ and $q \in \N$. Rewriting this condition in terms of $\lambda$, we find that 
\[|(q - p) \lambda - p\sqrt{1 - \lambda^2}| \ge \frac{c}{q^{1 + \beta}}\]
where $c$ is a possibly different constant. Put $q = k_1 + k_2$ and $p = k_2$ to find that 
\begin{equation}\label{eq:sq_diophantine_est}
    |k_1 \lambda - k_2 \sqrt{1 - \lambda^2}| \ge \frac{c}{q^{1 + \beta}} \gtrsim \frac{1}{|\mathbf k|^{1 + \beta}}.
\end{equation}
where the hidden constant is independent of $\mathbf k$. We clearly have 
\begin{equation}\label{eq:sq_est_2}
    |k_1 \lambda + k_2 \sqrt{1 - \lambda^2}| \ge 1,
\end{equation}
so combining (\ref{eq:sq_diophantine_est}) and (\ref{eq:sq_est_2}) yields
\begin{equation}\label{eq:combined_sq_est}
    |\lambda^2 k_1^2 - (1 - \lambda^2)k_2^2| \ge \frac{c}{|\mathbf{k}|^{1 + \beta}}.
\end{equation}
for a possibly different constant $c > 0$. Since $f \in \CIc(\Omega; \R)$, $\hat f(\mathbf k)$ is rapidly decreasing in $\mathbf k$, which tempers the denominator in (\ref{eq:square_solution}) to give uniform smoothness of $u(t)$ in time.  

The above analysis also exhibits the spectral result part (c) of Theorem \ref{thm:ergodic}. Note that $\kappa(\mathbf k)\sin(\pi k_1 x_1) \sin(\pi k_2 x_2)$, $\mathbf k \in \N^2$ form a complete orthonormal basis for $H^{-1}(\Omega)$, where $\kappa(\mathbf k) = \mathcal O(|k|)$ are simply normalizing constants. This basis consists of eigenfunctions with eigenvalues $\frac{k_1}{k_1 + k_2}$ for the operator $P$ defined in (\ref{eq:evolution_problem}). In particular, $\sin(\pi k_1 x_1) \sin(\pi k_2 x_2)$ has eigenvalue $\frac{k_2^2}{k_1^2 + k_2^2}$. If $\mathbf r(\lambda)$ is Diophantine, then (\ref{eq:combined_sq_est}) gives a characterization of the eigenvalues near $\lambda^2$:
\begin{equation*}
    \left|\frac{k_2^2}{k_1^2 + k_2^2} - \lambda^2\right| \le \epsilon \implies \frac{1}{|\mathbf k|^{3 + \beta}} \gtrsim \epsilon^{-1}
\end{equation*}
Therefore the spectral measure $\mu_{f, f}$ satisfies part (c) of Theorem \ref{thm:ergodic} near $\lambda^2$. Indeed, using the fact that the coefficients $\hat f(k)$ of $f \in \CIc(\Omega)$ defined in (\ref{eq:f_coeffs}) are rapidly decreasing, we have
\[\mu_{f, f}((\lambda^2 - \epsilon, \lambda^2 + \epsilon)) \lesssim \sum_{\mathbf k : \big|\frac{k_2}{k_1^2 + k_2^2} - \lambda^2 \big| \le \epsilon^{-1}} \hat f(\mathbf k)^2 \kappa(\mathbf k)^{-2} < C_d \epsilon^d \]
for any $d \in \N$. See Figure \ref{fig:spectral_bound}.

Finally, we mention that by tilting the square by $\eta$, the set of $\lambda$ for which $\mathbf \lambda$ is Diophantine is no longer a full measure set. More specifically, we consider the square domain specified by the vertices 
\[\big\{(0, 0), \, (\cos \eta, \sin \eta), \, (-\sin \eta, \cos \eta), \, \sqrt{2}(\cos(\eta + \tfrac{\pi}{4}), \sin(\eta + \tfrac{\pi}{4})) \big \}.\] 
Then graph of $\mathbf r(\lambda)$ is constant near values of $\lambda$ for which $\mathbf r(\lambda)$ is rational. See Figure~\ref{fig:rot_num} for an illustration and \cite[\S2.5]{DWZ} for details. 

\begin{figure}
    \centering
    \includegraphics[scale = 0.35]{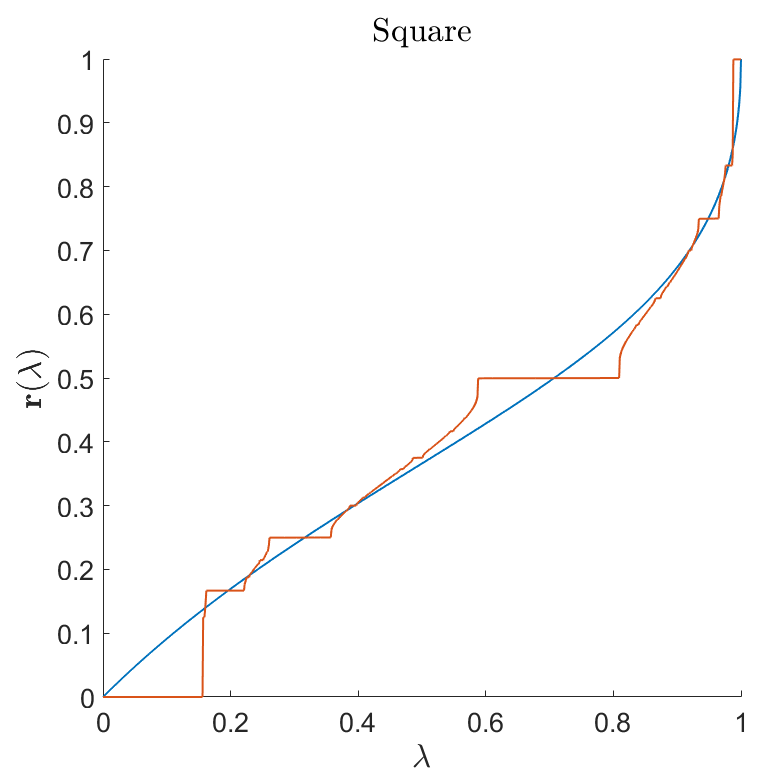}
    \includegraphics[scale = 0.35]{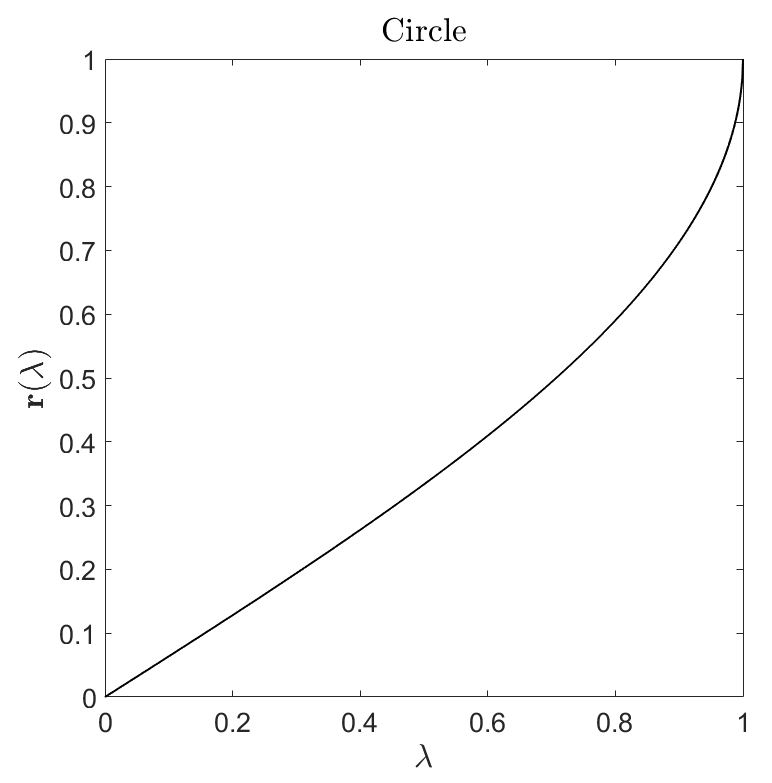}
    \caption{Graphs of the rotation number function $\mathbf r(\lambda)$ (defined in~\eqref{eq:rot_num_def}) for various domains. \textbf{Left:} The blue curve is for the untilted square $[0, 1]^2$. The orange curve is the square tilted by $\eta = \pi/20$. Note that it has rational plateaus. \textbf{Right:} The domain is a disk, so the boundary is a round circle.}
    \label{fig:rot_num}
\end{figure}

\subsection{The disk} \label{sec:disk}

We finally consider the case when $\Omega = \mathbb D$ is the unit disk. This has previously been
studied in~\cite[Section~9]{Aleksandryan_60}. The 3-dimensional case of a triaxial ellipsoid has been studied in~\cite{C-V}, following the physics work of~\cite{IJW15, BR17}.


We may assume $\mathbb D$ is centered at origin. Parameterize the boundary $\partial \mathbb D$ counterclockwise by arclength with $\phi = 0$ being the point $(1, 0)$. The boundary is then identified with the circle $\R/2 \pi \Z$. This is a different convention from the previous sections, and we switch conventions to avoid carrying factors of $2 \pi$. The mod 1 rotation number as defined in \ref{eq:rot_num_def} is given by 
\begin{equation}
    \mathbf r(\lambda) = 1 - \frac{2}{\pi} \arccos(\lambda),
\end{equation}
see \cite[\S 2.2.2]{Zhenhao} for details and Figure \ref{fig:rot_num}. It will be convenient to rescale the rotation number and define
\begin{equation}
    \alpha(\lambda) = \frac{\pi}{2} \mathbf r(\lambda).
\end{equation}
We henceforth drop $\lambda$ from the notation when there is no ambiguity. The terms of the factorization \eqref{eq:dual_factor} and the chess billiard map take the explicit forms
\begin{equation}\label{eq:circle_factor}
    \ell^\pm(x_1, x_2) = x_1 \cos \alpha \pm x_2 \sin \alpha, \qquad b(\phi) = \phi + 4 \alpha.
\end{equation}
We will provide an explicit complete basis of eigenfunctions. Recall from \eqref{eq:dual_factor} that $\ell^+(x, \lambda) \ell^-(x, \lambda)$ is dual to the symbol of $P(\lambda)$, so $u \in H^1_0(\Omega)$ is a solution to the eigenvalue problem if and only if
\begin{equation}\label{eq:pm_decomp}
    \begin{gathered}
        u(x) = u_+(\ell^+(x)) + u_-(\ell^-(x)) \quad \text{and} \\
        U_+ + U_- = 0 \quad \text{where} \quad U_+ := (u_+\circ \ell^+)|_{\partial \Omega}, \quad U_- := (u_- \circ \ell^-)|_{\partial \Omega}.
    \end{gathered}
\end{equation}
for some $u_\pm \in C^0$ and $u_{\pm} \circ \ell^\pm \in H^1(\Omega)$. See Lemma \ref{lem:decomp} for discussion of the regularity of $u_\pm$. Using the explicit formulas \eqref{eq:circle_factor}, we can compute 
\begin{equation}\label{eq:U_properties}
    \begin{gathered}
        U_\pm(\phi) = u_\pm(\cos \phi \cos \alpha \pm \sin \phi \sin \alpha) = u_\pm(\cos(\phi \mp \alpha)), \\
        U_\pm(\phi) = U_\pm(\phi + 4 \alpha)
    \end{gathered}
\end{equation}

Let $N \in \N$ and $k \in \{1, \dots, N - 1\}$. Then define
\[\alpha_{k,N} := \alpha(\lambda_{k, N}) = \frac{\pi}{2} \frac{k}{N} \quad \text{where} \quad
    \lambda_{k,N} \in (0, 1) \quad \text{uniquely satisfies} \quad \mathbf r(\lambda_{k, N}) = \frac{k}{N}.\]
We wish to construct a solution to every eigenvalue problem $P(\lambda_{k, N})u = 0$. Collecting the symmetries in \eqref{eq:pm_decomp} and \eqref{eq:U_properties}, we must have
\[U_+(-\phi + 2 \alpha_{k, N})=U_+(\phi)=U_+(-\phi + 2 \alpha_{k, N}), \quad U_- = - U_+.\]
$U_+$ is thus invariant under the action of a dihedral group, so we consider the Fourier modes on an interval of the boundary which is a fundamental domain, and we are led to functions of the form
\[U_\pm(\phi) = (\pm 1)^{k + 1} \cos(N(\phi \mp \alpha_{k, N})),\]
which precisely satisfies the conditions given in \eqref{eq:U_properties}. 
Let $T_N$ be the Chebyshev polynomials, which are defined by $T_N(\cos \phi) = \cos(N\phi)$. Therefore, we have solutions to the eigenvalue problem \eqref{eq:eig_problem} at $\lambda = \lambda_{k, N}$ given by 
\begin{equation}\label{basis}
\begin{gathered}
    u_{k, N}(x_1, x_2) = T_N(x_1 \cos \alpha_{k, N} + x_2 \sin \alpha_{k, N}) - (-1)^k T_N(x_1 \cos \alpha_{k, N} - x_2 \sin \alpha_{k, N}) \\
    \text{for every} \quad N \in \N, \quad k \in \{1, \dots, N - 1\}.
\end{gathered}
\end{equation}
The above forms a complete basis of $H_0^1(\mathbb D)$ since for every $N$, the set $\{u_{k, N}\}_{k = 1}^{N - 1}$ consists of $N - 1$ linearly independent degree $N$ polynomials vanishing at the boundary. Linear independence follows immediately from the fact that $\Delta u_{k, N}$ is an eigenfunction of $P$ with eigenvalue $\lambda_{k, N}$. See Figure \ref{fig:eigfuncs}. Furthermore, we see that there are infinitely many solutions to the eigenvalue problem \eqref{eq:eig_problem} for every $\lambda$ such that $\mathbf r(\lambda)$ is rational, since there are infinitely many ways to represent a rational number as $k/N$.

The explicit formulas \eqref{basis} for the solutions of the eigenvalue problem directly completes the proof of Theorem \ref{thm:ellipse} for the unit disk. The proof of Theorem \ref{thm:ellipse} for ellipses and rectangles is now a purely geometric problem. 
\begin{proof}[Proof of Theorem \ref{thm:ellipse}]
Let $\Omega$ be an ellipse. Then there exists a symmetric nondegenerate $2 \times 2$ matrix $A$ and $v \in \R^2$ such that $\Omega = A\mathbb D + v$ where $\mathbb D$ is the unit disk. Then observe that there exists a family of rotation matrices $R(\lambda)$ that depends smoothly on $\lambda$ such that under the coordinate change
\[y = R(\lambda) A^{-1}(x - v),\]
$P(\lambda)$ on $\Omega$ becomes $c(\lambda) P(\sigma(\lambda))$ on $\mathbb D$, for some $c(\lambda) > 0$ and $\sigma:[0, 1] \to [0, 1]$ smooth and monotonically increasing with $\sigma(0) = 0$ and $\sigma(1) = 1$. The result then follows from the explicit basis of eigenfunctions corresponding to eigenvalues dense in  $(0, 1)$ given in \eqref{basis} for the circle. 

The rectangle case follows from the square case simply by a linear change of coordinates by a diagonal matrix independent of $\lambda$.
\end{proof}

\medskip\noindent\textbf{Acknowledgements.}
The authors would like to thank Semyon Dyatlov, Maciej Zworski, and Leo Maas for insightful discussions. They would also like to thank Matthew Colbrook for providing the numerical data used in Figure \ref{fig:spectral_bound}. The second author is partially supported by Semyon Dyatlov's NSF CAREER grant DMS-1749858.

\bibliographystyle{alpha}
\bibliography{internal_waves.bib}

\end{document}